\documentclass[12pt,reqno]{amsart}
\usepackage{amsmath,amsfonts,graphicx,amscd,amssymb,epsf,amsthm,enumerate,parskip,mathrsfs,color,ltablex,blindtext}  
 \usepackage[a4paper, margin=1in]{geometry}
\usepackage{times}
\usepackage{comment}
\theoremstyle{definition}
\newtheorem{theorem}{Theorem}[section]
\newtheorem{proposition}[theorem]{Proposition}
\newtheorem{lemma}[theorem]{Lemma}

\newtheorem{definition}[theorem]{Definition}
\newtheorem{example}[theorem]{Example}

\newtheorem{remark}[theorem]{Remark}
 \numberwithin{equation}{section}
\numberwithin{equation}{section}
 
\newcommand{\di}{\displaystyle}

\definecolor{Purple}{RGB}{100,0,150}
\definecolor{Green}{RGB}{0,120,20}
\usepackage{cite}
\usepackage[colorlinks,citecolor=blue]{hyperref}
\newcommand{\fa}{$\hspace{0cm}$}
\begin{document}
  
\title[Rademacher's Formula]{  Rademacher's Formula for the Partition Function}
\author{Ze-Yong Kong and Lee-Peng Teo}
\address{Department of Mathematics, Xiamen University Malaysia\\Jalan Sunsuria, Bandar Sunsuria, 43900, Sepang, Selangor, Malaysia.}
\email{MAT2004438@xmu.edu.my, lpteo@xmu.edu.my}
\begin{abstract}
For a positive integer $n$, let $p(n)$ be the number of ways to express $n$ as a sum of positive integers. 
In this note, we revisit the  derivation of the Rademacher's convergent series for $p(n)$ in a pedagogical way, with all the   details given. We also derive the leading asymptotic behavior of $p(n)$ when $n$ approaches infinity. Some numerical results are tabled.
\end{abstract}
\maketitle
\section{Introduction}
Given a positive integer $n$, let $p(n)$ be the number of ways to express $n$ as a sum of positive integer. 
For example, when $n=7$,
the followings are the different ways to express 7 as a sum of positive integers.
\begin{gather*}
\textcolor{red}{7}
\\
\textcolor{Purple}{6+1}, \;\textcolor{Green}{ 5+2}, \;\textcolor{Purple}{4+3}\\
\textcolor{red}{5+1+1},\; \textcolor{blue}{4+2+1}, \;\textcolor{red}{3+3+1},\; \textcolor{blue}{3+2+2}\\
\textcolor{Purple}{4+1+1+1}, \;\textcolor{Green}{3+2+1+1}, \;\textcolor{Purple}{2+2+2+1}\\
\textcolor{red}{3+1+1+1+1}, \;\textcolor{blue}{2+2+1+1+1}\\
\textcolor{Purple}{2+1+1+1+1+1}\\
\textcolor{Green}{1+1+1+1+1+1+1}
\end{gather*}There are 15 of them. Hence, $p(7)=15$.

Let $p(0)=1$. Then the generating function of the sequence $\{p(n)\}$ is
\begin{align}\label{eq3_23_4}F(x)=\sum_{n=0}^{\infty}p(n)x^n=\prod_{m=1}^{\infty}(1-x^m)^{-1}.\end{align}

The calculation of $p(n)$ using the generating function is inefficient. A more efficient method is to make use of the Euler pentagonal number theorem, which states that 
\[\prod_{m=1}^{\infty}(1-x^m)=1+\sum_{n=1}^{\infty}(-1)^n\left(x^{\omega_1(n)}+x^{\omega_2(n)}\right),\]
where
\[\omega_1(n)=\sum_{k=0}^{n-1}(1+3k)=\frac{3n^2-n}{2} \] are the pentagonal numbers, and \[\omega_2(n)=\omega_1(-n)=\frac{3n^2+n}{2}.\]

From this, one can deduce the following recursive formula to calculate $p(n)$:
\begin{equation}\label{eq230113_1}
p(n)=\sum_{\substack{k\geq 1\\w_2(k)\leq n}}(-1)^{k+1}\left\{p(n-\omega_1(k)+p(n-\omega_2(k)\right\}.\end{equation}
Notice that for all $n\geq 1$,
\[\omega_2(n)-\omega_1(n)=n,\hspace{1cm}\omega_1(n+1)-\omega_2(n)=2n+1.\]
This implies that
\[\omega_1(n)<\omega_2(n)<\omega_1(n+1).\]The sum on the right hand side of \eqref{eq230113_1} only involves approximately $\di 2\sqrt{2n/3}$ terms. Hence, it is an effective formula to calculate $p(n)$. By using $\omega_1(k)$ and $\omega_2(k)$ up to $k=100$, one can calculate all $p(n)$ for $1\leq n\leq 15050$. Selected values of $p(n)$ for $n$ in this range  are tabled in Table \ref{table2}.

In 1918, Hardy and Ramanujan \cite{Hardy1918} showed that $p(n)$ satisfies the asymptotic formula:
$$p(n)\sim \frac{1}{4n\sqrt{3}}\exp\left(\pi\sqrt{\frac{2n}{3}}\right)\hspace{1cm} \text{when}\;n\rightarrow\infty.$$ More precisely, if we let
\begin{align}\label{eq2_15_1}P_1(n)= \frac{1}{4n\sqrt{3}}\exp\left(\pi\sqrt{\frac{2n}{3}}\right),\end{align} then
$$\lim_{n\rightarrow\infty}\frac{p(n)}{P_1(n)}=1.$$This was discovered independently by Uspensky \cite{Uspensky1920} in 1920. 

In fact, in \cite{Hardy1918}, Hardy and Ramanujan obtained the following aymptotic formula
\begin{align}\label{eq2_15_2}p(n)=\sum_{k<\alpha\sqrt{n}}P_k(n)+O\left(n^{-1/4}\right),\end{align} where $\alpha$ is a positive constant, and $P_1(n)$ is the leading term given by \eqref{eq2_15_1}. The terms $P_2(n)$, $P_3(n)$, $\cdots$, have similar forms, but with a constant smaller than $\di\pi\sqrt{2/3}$ in the exponential term. 

A drawback of the asymptotic formula \eqref{eq2_15_2} is that the infinite series
$$\sum_{k=1}^{\infty} P_k(n)$$ is divergent, as was proved by Lehmer \cite{Lehmer1937}.

In 1937, when  Rademacher  prepared lecture notes on the work of Hardy and Ramanujan, he made an improvement on the asymptotic formula \eqref{eq2_15_2}. He obtained the following remarkable formula.

\begin{theorem}\label{thmrademacher}[\textbf{Rademacher's Formula \cite{Rademacher1937}}]~

If $n\geq 1$, the partition function $p(n)$ is represented by the convergent series
\begin{align*}
p(n)=\sum_{k=1}^{\infty}\frac{1}{\pi\sqrt{2}}A_k(n)\sqrt{k}\frac{d}{dn}\left(\frac{\di\sinh\left\{\frac{\pi}{k}\sqrt{\frac{2}{3}\left(n-\frac{1}{24}\right)}\right\}}{\di\sqrt{n-\frac{1}{24}}}\right),
\end{align*}where $A_k(n)$ is defined in Theorem \ref{thmrademacher2}.
 
\end{theorem}

In contrast to the Hardy-Ramanujan formula, Rademacher's formula is a convergent series. In \cite{Rademacher1937}, Rademacher also showed that if $N$ is of order $\sqrt{n}$, then the remainder after $N$ terms is of order $n^{-1/4}$.

The Rademacher's formula is a manifestation of the achievement of the circle method of Hardy, Ramanujan and Littlewood. This method has been very successful in problems of additive number theory \cite{Rademacher1940, Rademacher1943, Vaughan1981}. 

Partition function and its generalizations  have been under active studies \cite{Grosswald1958, Grosswald1960, Grosswald1960_2, Grosswald1962, Grosswald1963, Grosswald1984, Hagis1962, Hagis1963, Hagis1964, Hagis1964_2, Hagis1963, Hagis1965, Hagis1965_2, Hagis1965_3, Hagis1966, Hagis1970,Hagis1971,Hagis1971_2, Hua1942,  Iseki1959,  Iseki1960_2,   Livingood1945, Niven1940, Robertson1976, Robertson1976_2, Robertson1977, Selberg1989, Spencer1973, Subrahmanyasastri1972, Vangipuram1982, Vangipuram1982_2}, especially in recent years \cite{Almkvist1995, Barber2013, Brassesco2020, Bringmann2006, Craig_2022, Dewar2013, Iskander2020, Jennings2015, Johansson2012, Khaochim2019, Kim1999, McLaughlin2012, OSullivan2020, Pribitkin2019,
Selberg1989, Sills2010, Sills2010_2, Sills2010_3, Sills2012, Pribitkin2000, Pribitkin2009, Pribitkin2019,  Tani2011, Zagier2021}. Modern approach to proving Rademacher's formula usually involves advanced mathematics such as modular forms or Poincar$\acute{\text{e}}$ series \cite{Dewar2013, Pribitkin2000, Pribitkin2009, Pribitkin2019}.

The purpose if this note is to present the proof of the Rademacher's formula of partition function with all the necessary details. We follow closely the approach in \cite{Apostol_2}.

\bigskip
\section{Preliminaries}

\subsection{Dedekind Eta Function}
The Dedekind eta function is introduced by Dedekind in 1877 and is defined in the upper half plane $\di \mathbb{H}=\left\{\tau\,|\,\text{Im}\,\tau>0\right\}$ by the equation
\begin{align}\label{eq3_31_8}
\eta(\tau)=e^{\pi i\tau/12}\prod_{n=1}^{\infty}\left(1-e^{2\pi i n\tau}\right).
\end{align}The generating function $F(x)$ \eqref{eq3_23_4} for the sequence $\{p(n)\}$ is related to $\eta(\tau)$ by
\begin{align}
\label{eq3_23_5}
\eta(\tau)=e^{\pi i\tau/12}F(e^{2\pi i\tau}).
\end{align}The eta function $\eta(\tau)$ is closely related to the theory of modular forms \cite{Apostol_2}. 
Its 24th power, $\eta(\tau)^{24}$, is a modular form of weight 12.

To prove  Rademacher's formula, one needs a transformation formula for the Dedekind eta function under a fractional linear transformation
\[\tau \mapsto \frac{a\tau+b}{c\tau+d} \] 
defined by an element $\di \begin {bmatrix} a & b\\ c & d\end{bmatrix}$   of the modular group  $\Gamma=\text{PSL}\,(2,\mathbb{Z})$.
  
Before stating the transformation formula, let us define the Dedekind sum. 

\begin{definition}
If $h$ is an integer, $k$ is a positive integer larger than 1, the Dedekind sum $s(h,k)$ is defined as
\begin{gather}\label{eq8_13_1}
s(h,k)=\sum_{r=1}^{k-1}\frac{r}{k}\left(\frac{hr}{k}-\left\lfloor\frac{hr}{k}\right\rfloor-\frac{1}{2}\right).
\end{gather}When $k=1$,   
 $s(h,1) $ is defined as 0 for any integer $h$. 
\end{definition}

\begin{theorem}
[The Dedekind's Functional Equation]\label{thm4_5_2}~\\
Let $a, b, c, d$ be integers with $ad-bc=1$ and $c>0$. Under the fractional linear transformation 
\[\tau\mapsto\frac{a\tau+b}{c\tau+d},\]the Dedekind eta function satisfies the transformation formula.
\begin{equation}\label{eq8_13_6}
\eta\left(\frac{a\tau+b}{c\tau+d}\right)=\exp\left\{\pi i\left(\frac{a+d}{12c}+s(-d,c)\right)\right\}\left\{-i(c\tau+d)\right\}^{1/2}\eta(\tau).
\end{equation}
\end{theorem}
A proof of this theorem is given in \cite{Apostol_2} using a more general formula proved by Iseki \cite{Iseki_1957}. In \cite{Kong_Teo_1}, we use a simpler approach to derive that formula.

\bigskip
\subsection{Modified Bessel Functions}
Let $\nu$ be a complex number. The modified Bessel function $I_{\nu}(x)$ is defined as
\begin{align}\label{eq4_6_1}
I_{\nu}(x)=\left(\frac{x}{2}\right)^{\nu}\sum_{j=0}^{\infty}\frac{1}{j!\,\Gamma(\nu+j+1)}\left(\frac{x}{2}\right)^{2j}.
\end{align} $I_{\nu}(x)$ has a contour integral representation given by
\[I_{\nu}(x)=\left(\frac{x}{2}\right)^{\nu}\frac{1}{2\pi i}\int_{c-i\infty}^{c+i\infty}t^{-\nu-1}\exp\left(t+\frac{x^2}{4t}\right)dt,\] where $c$ is a positive number. To prove this, we need the Hankel's formula.

Let $K$  be the contour which is a loop around the negative real axis. As shown in Figure \ref{figure1}, it consists of three parts $K_1$, $K_2$ and $K_3$, where $K_1$ and $K_3$ are the lower and upper edges of a cut in the $z$-plane along the negative real-axis, and $K_2$ is a positively oriented circle of radius $c<2\pi$ about the origin.

\begin{figure}[h]
\centering
\includegraphics[scale=0.58]{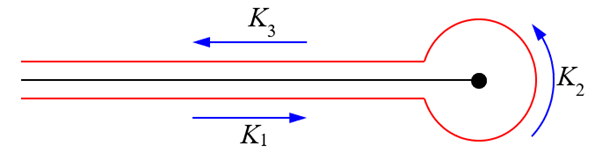}

\caption{The contour $K$.\fa}\label{figure1}
\end{figure}
\begin{lemma}[Hankel's Formula]\label{lemma4_12_1}
When $s\in\mathbb{C}$, we have
\begin{align*}
\frac{1}{\Gamma(s)}=\frac{1}{2\pi i}\int_K z^{-s}e^zdz.
\end{align*}

\end{lemma}
\begin{proof}
The integral defines an analytic function for all $s\in \mathbb{C}$. It is sufficient to prove the formula for $s$ with $s<0$. Then the result follows from analytic continuation.
When $s<0$, the integral over $K_2$ goes to $0$ as $c\rightarrow 0^+$. In this limit, we have
\begin{align*}
\frac{1}{2\pi i}\int_K z^{-s}e^zdz&=-e^{-\pi i}\frac{1}{2\pi i}\int_0^{\infty}(e^{-\pi i}t)^{-s}e^{-t}dt+e^{\pi i}\frac{1}{2\pi i}\int_0^{\infty}(e^{\pi i}t)^{-s}e^{-t}dt\\
&=\frac{e^{\pi i s}-e^{-\pi i s}}{2\pi i}\int_0^{\infty}t^{-s}e^{-t}dt\\
&=\frac{\sin \pi s}{\pi}\Gamma(1-s)\\
&=\frac{1}{\Gamma(s)}.
\end{align*}

\end{proof}
Using this lemma, we can prove the contour integral formula for the modified Bessel function $I_{\nu}(x)$.

\begin{proposition}
Let  $\nu$ be a complex number and let $c$ be a positive number. If $x$ is real, then
\begin{align}\label{eq4_6_2}
I_{\nu}(x)=\left(\frac{x}{2}\right)^{\nu}\frac{1}{2\pi i}\int_{c-i\infty}^{c+i\infty}t^{-\nu-1}\exp\left(t+\frac{x^2}{4t}\right)dt.
\end{align}
\end{proposition}
\begin{proof}The integrand 
\[t^{-\nu-1}\exp\left(t+\frac{x^2}{4t}\right)\]is an analytic function of $t$ on the domain $\mathbb{C}-\left\{t\,|\,\text{Im}\,t=0,\;\text{Re}\,t\leq 0\right\}$. It decays exponentially fast in the region $\text{Re}\,t\leq c$ when $|t|\rightarrow\infty$.
Using Cauchy residue theorem, we can take the limit $c\rightarrow 0^+$, and change the contour to $K$. Namely,
\begin{align*}
f(x)&=\left(\frac{x}{2}\right)^{\nu}\frac{1}{2\pi i}\int_{c-i\infty}^{c+i\infty}t^{-\nu-1}\exp\left(t+\frac{x^2}{4t}\right)dt\\&=\left(\frac{x}{2}\right)^{\nu}\frac{1}{2\pi i}\int_{K}t^{-\nu-1}\exp\left(t+\frac{x^2}{4t}\right)dt.
\end{align*}Now, using the Taylor series of the exponential function, we have
\begin{align*}
f(x)&=\left(\frac{x}{2}\right)^{\nu}\sum_{j=0}^{\infty}\frac{1}{j!}\left(\frac{x}{2}\right)^{2j}\frac{1}{2\pi i}\int_{K}t^{-\nu-j-1}e^tdt\\
&=\left(\frac{x}{2}\right)^{\nu}\sum_{j=0}^{\infty}\frac{1}{j!}\left(\frac{x}{2}\right)^{2j}\frac{1}{\Gamma(\nu+j+1)},
\end{align*}where we have used Lemma \ref{lemma4_12_1}.
It follows from \eqref{eq4_6_1} that
$$f(x)=I_{\nu}(x).$$

\end{proof}

For our applications, we are interested in the case where $\nu=3/2$.
\begin{proposition}\label{prop4_13_1}
 $I_{\frac{3}{2}}(x)$ has an explicit formula given by
\begin{align*}
I_{\frac{3}{2}}(x)=\sqrt{\frac{2x}{\pi}} \frac{d}{dx}  \frac{\sinh x}{x}.
\end{align*}
\end{proposition}
\begin{proof}
This is a direct computation using \eqref{eq4_6_1}. Using  the fact that
$$\Gamma\left(j+\frac{5}{2}\right)=\frac{(2j+3)(2j+1)\ldots\times 1}{2^{j+2}}\Gamma\left(\frac{1}{2}\right)=\sqrt{\pi}\frac{(2j+3)(2j+1)\ldots\times 1}{2^{j+2}},$$ we have
\begin{align*}
I_{\frac{3}{2}}(x)&=\left(\frac{x}{2}\right)^{\frac{3}{2}}\sum_{j=0}^{\infty}\frac{1}{j!\,\Gamma( j+\frac{5}{2})}\left(\frac{x}{2}\right)^{2j}\\
&=\frac{2}{\sqrt{\pi}}\left(\frac{x}{2}\right)^{\frac{1}{2}}\sum_{j=0}^{\infty}\frac{2j+2}{(2j+3)!}x^{2j+1}\\
&=\sqrt{\frac{2x}{\pi}} \frac{d}{dx}\sum_{j=0}^{\infty}\frac{1}{(2j+3)!}x^{2j+2}\\
&=\sqrt{\frac{2x}{\pi}} \frac{d}{dx}\sum_{j=0}^{\infty}\frac{1}{(2j+1)!}x^{2j}\\
&=\sqrt{\frac{2x}{\pi}} \frac{d}{dx}  \frac{\sinh x}{x}.
\end{align*}
\end{proof}

\bigskip
\subsection{Farey Series and Ford Circles}\label{secFarey}
 Let $n$ be a positive integer. The set of Farey fractions of order $n$, denoted by $F_n$, is the set of all reduced fractions in the closed interval $[0,1]$ whose denominator is not larger than $n$, listed in increasing order of magnitude. 
 
 \begin{example}
 The sets $F_n$, $1\leq n\leq 10$, are given by
 \begin{align*}
 F_1:&\; \frac{0}{1}, \frac{1}{1}\\
 F_2:&\;\frac{0}{1}, \frac{1}{2}, \frac{1}{1}\\
F_3:& \frac{0}{1}, \frac{1}{3}, \frac{1}{2}, \frac{2}{3}, \frac{1}{1}\\
F_4:& \frac{0}{1},\frac{1}{4}, \frac{1}{3}, \frac{1}{2}, \frac{2}{3},\frac{3}{4}, \frac{1}{1}\\
F_5:& \frac{0}{1},\frac{1}{5}, \frac{1}{4}, \frac{1}{3}, \frac{2}{5}, \frac{1}{2}, \frac{3}{5},\frac{2}{3},\frac{3}{4},\frac{4}{5}, \frac{1}{1}\\
F_6:& \frac{0}{1},\frac{1}{6},\frac{1}{5}, \frac{1}{4}, \frac{1}{3}, \frac{2}{5}, \frac{1}{2}, \frac{3}{5},\frac{2}{3},\frac{3}{4},\frac{4}{5}, \frac{5}{6},\frac{1}{1}\\
F_7:& \frac{0}{1},\frac{1}{7},\frac{1}{6},\frac{1}{5}, \frac{1}{4}, \frac{2}{7},\frac{1}{3}, \frac{2}{5}, \frac{3}{7},\frac{1}{2}, \frac{4}{7},\frac{3}{5},\frac{2}{3},\frac{5}{7},\frac{3}{4},\frac{4}{5}, \frac{5}{6},\frac{6}{7},\frac{1}{1}\\
F_8:& \frac{0}{1},\frac{1}{8},\frac{1}{7},\frac{1}{6},\frac{1}{5}, \frac{1}{4}, \frac{2}{7},\frac{1}{3}, \frac{3}{8},\frac{2}{5}, \frac{3}{7},\frac{1}{2}, \frac{4}{7},\frac{3}{5},\frac{5}{8},\frac{2}{3},\frac{5}{7},\frac{3}{4},\frac{4}{5}, \frac{5}{6},\frac{6}{7},\frac{7}{8},\frac{1}{1}\\
F_9:& \frac{0}{1},\frac{1}{9},\frac{1}{8},\frac{1}{7},\frac{1}{6},\frac{1}{5}, \frac{2}{9},\frac{1}{4}, \frac{2}{7},\frac{1}{3}, \frac{3}{8},\frac{2}{5}, \frac{3}{7},\frac{4}{9},\frac{1}{2}, \frac{5}{9},\frac{4}{7},\frac{3}{5},\frac{5}{8},\frac{2}{3},\frac{5}{7},\frac{3}{4},\frac{7}{9},\frac{4}{5}, \frac{5}{6},\frac{6}{7},\frac{7}{8},\frac{8}{9},\frac{1}{1}\\
F_{10}:& \frac{0}{1},\frac{1}{10},\frac{1}{9},\frac{1}{8},\frac{1}{7},\frac{1}{6},\frac{1}{5}, \frac{2}{9},\frac{1}{4}, \frac{2}{7},\frac{3}{10},\frac{1}{3}, \frac{3}{8},\frac{2}{5}, \frac{3}{7},\frac{4}{9},\frac{1}{2}, \frac{5}{9},\frac{4}{7},\frac{3}{5},\frac{5}{8},\frac{2}{3},\frac{7}{10},\frac{5}{7},\frac{3}{4},\frac{7}{9},\frac{4}{5}, \frac{5}{6},\frac{6}{7},\frac{7}{8},\frac{8}{9},\frac{9}{10},\frac{1}{1}
 \end{align*}
 \end{example}
 Obviously, for $n\geq 2$, $F_{n-1}$ is a subset of $F_{n}$. Moreover,
 $$|F_{n}|-|F_{n-1}|=\varphi(n),$$ which is the number of positive integers less than or equal to $n$ that are relatively prime to $n$.
 
 To be more precise, one adds in $\varphi(n)$ fractions of the form $\di\frac{h}{n}$ with $1\leq h\leq n$ and $(h,n)=1$ into the set $F_{n-1}$ to obtain $F_n$. To determine the precise places to insert these fractions, one need the following lemmas.  
 
 \begin{lemma}\label{lemma3_16_2}
 If $a$, $b$, $c$ and $d$ are positive integers  such that $$\di \frac{a}{b}<\frac{c}{d},$$ then  
 $$\frac{a}{b}<\frac{a+c}{b+d}<\frac{c}{d}.$$
 
 \end{lemma}
 \begin{proof}
 We have
 \begin{gather*}
 \frac{c}{d}-\frac{a}{b}=\frac{bc-ad}{bd}>0
 \end{gather*}Since $b$ and $d$ are positive, we have
 $$bc-ad>0.$$
 Hence,
 \begin{align*}
 \frac{a+c}{b+d}-\frac{a}{b}&=\frac{bc-ad}{b(b+d)}>0,\\
 \frac{c}{d}-\frac{a+c}{b+d}&=\frac{bc-ad}{d(b+d)}>0.
 \end{align*}This proves the assertion.
 \end{proof}
 
 \begin{lemma}\label{lemma3_16_1}
 If $a$, $b$, $c$, $d$, $h$ and $k$ are positive integers such that
 \begin{align}\label{eq3_16_1}
\frac{a}{b}<\frac{h}{k}<\frac{c}{d}
\end{align} and $bc-ad=1$,  then
 $$k\geq b+d.$$
 \end{lemma}
\begin{proof}
The   inequalities in \eqref{eq3_16_1}  imply that
\begin{align*}
bh-ak&\geq 1,\\
ck-dh&\geq  1.
\end{align*}
Therefore,
\begin{align*}
k&=  k(bc-ad)\\
&=b(ck-dh)+d(bh-ak)\\
&\geq  b+d.
\end{align*}
\end{proof}

From this, we can deduce the following theorem.
\begin{theorem}\label{thm3_16_1}
Let $a$, $b$, $c$, $d$  be positive integers such that
 \begin{align*} 
\frac{a}{b}< \frac{c}{d},
\end{align*} and 
$bc-ad=1$. If 
$$\max\{b, d\}\leq n\leq b+d-1,$$ then $a/b$ and $c/d$ are consecutive terms in $F_n$.

\end{theorem}
\begin{proof}
Since $\max\{b,d\}\leq n$, $b$ and $d$ are less than or equal to $n$. Hence, $a/b$ and $c/d$ are in $F_n$.\\
 By Lemma \ref{lemma3_16_1}, no fractions $h/k$ with $k\leq b+d-1$ can satisfy
$$\frac{a}{b}<\frac{h}{k}<\frac{c}{d}.$$Since $F_n$ only contains reduced fractions with denominator at most $n$ and $n\leq b+d-1$, this shows that $a/b$ and $c/d$ are consecutive terms in $F_n$.
\end{proof}

Now we can give the precise algorithm to construct $F_n$ from $F_{n-1}$.
First notice that for the two terms $$\frac{a}{b}=\frac{0}{1}\quad \text{and} \quad\frac{c}{d}=\frac{1}{1}$$ in $F_1$,  $a=0$, $b=c=d=1$ and thus
$$bc-ad=1.$$
Now to construct $F_2$, we insert the term $$\frac{h}{k}=\frac{a+c}{b+d}=\frac{1}{2}$$ in between them.
It is easy to check that if $bc-ad=1$, $h=a+c$ and $k=b+d$, then
$$bh-ak=1\quad\text{and}\quad ck-dh=1.$$ In general, assume that $F_{n-1}$ has been constructed and for any two consecutive terms $a/b$ and $c/d$ in $F_{n-1}$, $bc-ad=1$. 
To construct $F_n$, we need  to insert the $\varphi(n)$ fractions $h/n$ with $1\leq h\leq n$ and $(h,n)=1$ into $F_{n-1}$. To do this, we look for consecutive pairs of reduced fractions $a/b$ and $c/d$ in $F_{n-1}$ for which $a/b<c/d$ and $b+d=n$. Let $h=a+c$. Since $0\leq a\leq b$ and $0\leq c\leq d$, we have
$$0\leq h=a+c\leq b+d.$$  Lemma \ref{lemma3_16_2} implies that
 $$\frac{a}{b}<\frac{h}{n}<\frac{c}{d}.$$ By induction hypothesis, $bc-ad=1$. As we have shown, this implies that 
$$bh-an=1\quad\text{and}\quad cn-dh=1.$$
It remains to show that we can find exactly $\varphi(n)$ pairs of consecutive terms $a/b$ and $c/d$ in $F_{n-1}$ such that $b+d=n$.

For $n\geq 2$, there are exactly $\varphi(n)$ positive integers $b$ with $1\leq b\leq n$ and $(b,n)=1$. In fact, we must also have $b\leq n-1$. For   each of these integers $b$, $d=n-b$ is also relatively prime to $n$ and $1\leq d\leq n-1$. Since $b+d=n$, $b$ and $d$ must also be relatively prime. There exists $s_0$ and $t_0$ such that
$$bs_0-dt_0=1.$$ Any  pairs of integers $(s,t)$ with $bs -dt=1$ can be written as
\begin{align*}
s=s_0+du,\quad t=t_0+bu
\end{align*}for some integer $u$. Let $u_0$ be the smallest one so that $c=s_0+du_0\geq 1$. Then $c\leq d$. If $a=t_0+bu_0$, then $bc-ad=1$. Since 
$$ad=bc-1<bc\leq bd,$$ we find that $a<b$. Since $b\geq 1, c\geq 1$, we have $ad\geq 0$ and thus $a\geq 0$. In other words, we have shown that for each of the $\varphi(n)$ positive integers $b$ so that $1\leq b\leq n$ and $(b,n)=1$, there are unique integers $a, c$ and $d$ such that
$$0\leq \frac{a}{b}<\frac{c}{d}\leq 1$$ and $bc-ad=1$. This shows that we can find exactly $\varphi(n)$ consecutive pairs in $F_{n-1}$ for us to insert the fractions $h/n$ with $1\leq h\leq n$ and $(h,n)=1$.

Given a reduced fraction $h/k$, the Ford circle $C(h,k)$ (Figure \ref{figure2}) is the circle in the complex plane with center at $$c_{h,k}=\frac{h}{k}+\frac{i}{2k^2}$$ and radius
$$r_{h,k}=\frac{1}{2k^2}.$$
Obviously, each Ford circle $C(h,k)$ touches the $x$-axis at the point $(h/k, 0)$.

\begin{figure}[h]
\centering
\includegraphics[scale=0.9]{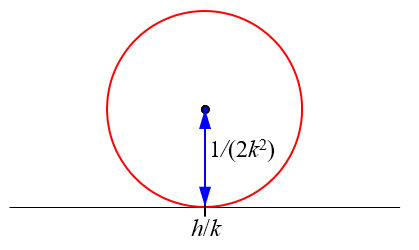}

\caption{The Ford circle $C(h,k)$.\fa}\label{figure2}
\end{figure}

\begin{theorem}
Two distinct Ford circles $C(a,b)$ and $C(c,d)$ are either tangent to each other or disjoint. They are tangent if and only if $bc-ad=\pm 1$. In particular, if $a/b$ and $c/d$ are two consecutive Farey fractions, then the two Ford circles $C(a,b)$ and $C(c,d)$ are tangent to each other.
\end{theorem}
\begin{proof}

\begin{figure}[h]
\centering
\includegraphics[scale=0.9]{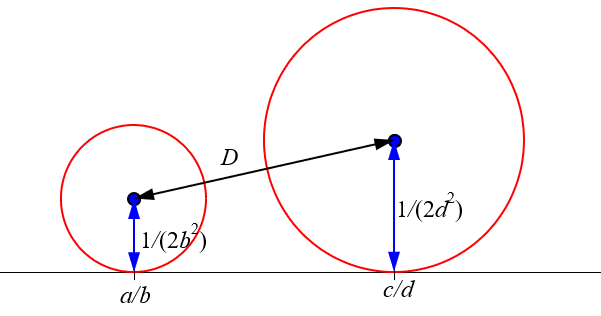}

\caption{The circles $C(a,b)$ and $C(c,d)$.\fa}\label{figure3}
\end{figure}

Let $D$ be the distance between the centers of $C(a,b)$ and $C(c,d)$. Then
\begin{align*}
D^2=\left(\frac{a}{b}-\frac{c}{d}\right)^2+\left(\frac{1}{2b^2}-\frac{1}{2d^2}\right)^2.
\end{align*}The sum of the radii of the two circles is
$$S=\frac{1}{2b^2}+\frac{1}{2d^2}.$$
Notice that   the two circles are disjoint if and only if $D>S$,  and the two circles are tangent to each other if and only if $D=S$.
A straightforward computation gives
\begin{align*}
D^2-S^2&=\frac{(bc-ad)^2-1}{b^2d^2}.
\end{align*}The two circles are distinct, so  $bc-ad\neq 0$. Therefore, $(bc-ad)^2\geq 1$. This shows that
$$D^2\geq S^2$$ and so $D\geq S$. Moreover, $D=S$ if and only if $(bc-ad)^2=1$, if and only if $bc-ad=\pm 1$. 
\end{proof}

Now we study the points of tangency of two Ford circles that are tangent to each other.
\begin{theorem}
Let $$\frac{h_1}{k_1}<\frac{h}{k}<\frac{h_2}{k_2}$$ be three consecutive Farey fractions. The points of tangency of $C(h,k)$ with $C(h_1, k_1)$ and $C(h_2, k_2)$ are given by
\begin{subequations}\label{eq3_16_5}
\begin{align}
\alpha_1(h,k)&=\frac{h}{k}-\frac{k_1}{k(k^2+k_1^2)}+\frac{i}{k^2+k_1^2},\label{eq3_16_5_1}\\
\alpha_2(h,k)&=\frac{h}{k}+\frac{k_2}{k(k^2+k_2^2)}+\frac{i}{k^2+k_2^2}.\label{eq3_16_5_2}
\end{align}\end{subequations}
\end{theorem}
\begin{proof}Notice that the condition of tangency implies that
$$hk_1-kh_1=1,\quad kh_2-hk_2=1.$$

\begin{figure}[h]
\centering
\includegraphics[scale=0.9]{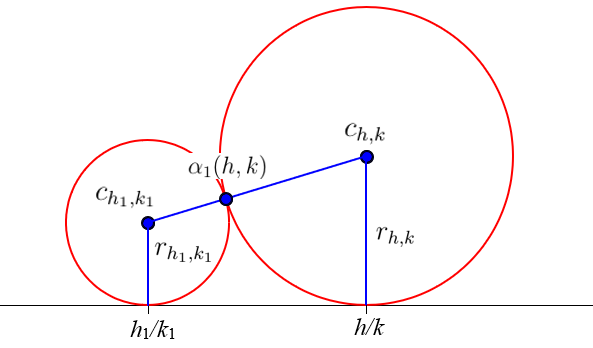}

\caption{The circles $C(a,b)$ and $C(c,d)$.\fa}\label{figure4}
\end{figure}

As shown in Figure \ref{figure4}, the point of tangency $\alpha_1(h,k)$ between $C(h_1, k_1)$ and $C(h,k)$ is the unique point on the line segment joining $c_{h_1, k_1}$ and $c_{h,k}$ such that 
$$\frac{c_{h,k}-\alpha_1(h,k)}{c_{h,k}-c_{h_1, k_1}}=\frac{r_{h,k}}{r_{h_1, k_1}+r_{h,k}}.$$ A direct computation proves \eqref{eq3_16_5_1}. 
Similarly, $\alpha_2(h,k)$ is the unique point on the line segment joining $c_{h,k}$ to $c_{h_2,k_2}$ such that
$$\frac{\alpha_2(h,k)-c_{h,k}}{c_{h_2,k_2}-c_{h, k}}=\frac{r_{h,k}}{r_{h_2, k_2}+r_{h,k}}.$$ 
This gives \eqref{eq3_16_5_2}.
\end{proof}

For each positive integer $N$, we construct a path  $P(N)$ joining the points $i$ and $1+i$ in the complex plane as follows. Let
$$0, \frac{h_1}{k_1}, \ldots, \frac{h_{m_N-1}}{k_{m_N-1}}, 1$$ be the Farey fractions in $F_N$ listed in increasing order. Let $$\frac{h_0}{k_0}=0\quad \text{and} \quad \frac{h_{m_N}}{k_{m_N}}=1.$$ For each $1\leq j\leq m_N-1$, let $A_j$ be the point of tangency between $C(h_{j-1},k_{j-1})$ and $C(h_j, k_j)$; and let $B_j$ be the point of tangency between $C(h_j, k_j)$ and $C(h_{j+1}, k_{j+1})$. Then $B_j=A_{j+1}$ for $1\leq j\leq m_N-2$. $A_j$ and $B_j$ divide the circle $C(h_j, k_j)$ into two arcs, the upper arc and the lower arc. Let $U_j$ be the upper arc. Finally, let $U_0$ be the arc of the circle $C(h_0, k_0)$ that joins the point $i$ to $A_1$ that is in the right half plane $\text{Re}\,z\geq 0$, and let $U_{m_N}$ be the arc of the circle $C(h_{m_N}, k_{m_N})$ that joins $B_{m_N-1}$ to the point $1+i$ that is in the left half-plane $\text{Re}\;z\leq 1$.
The path $P(N)$ is defined as
\begin{align}
\label{eq3_31_10}
P(N)=\bigcup_{j=0}^{m_N}U_j,
\end{align} following the orientation from the point $i$ to the point $1+i$.

\begin{figure}[h]
\centering
\includegraphics[scale=0.9]{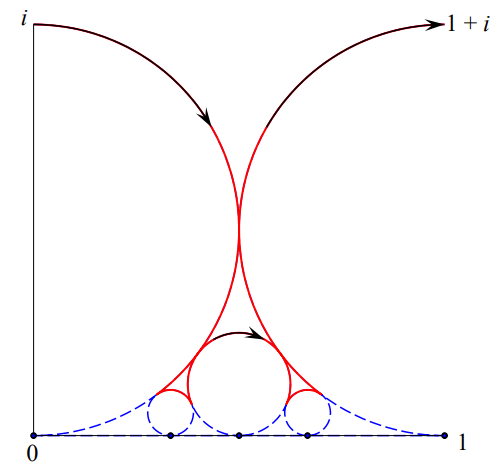}

\caption{The path $P(3)$ for the Farey series $\di F_3:  \frac{0}{1}, \frac{1}{3},   \frac{1}{2},  \frac{2}{3},  \frac{1}{1}$.\fa}\label{figure8}
\end{figure}

\vfill\pagebreak
\begin{figure}[h]
\centering
\includegraphics[scale=0.9]{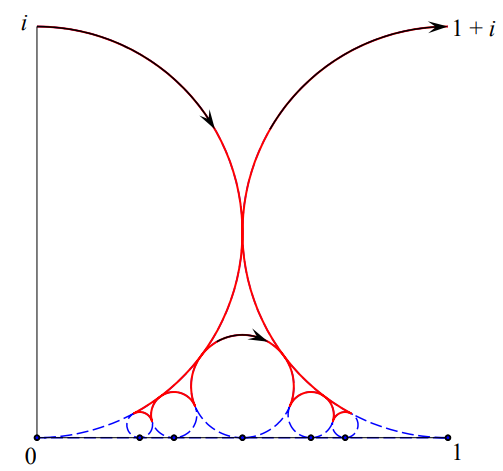}

\caption{The path $P(4)$ for the Farey series $\di F_4:  \frac{0}{1},  \frac{1}{4}, \frac{1}{3},   \frac{1}{2},  \frac{2}{3},\frac{3}{4},  \frac{1}{1}$.\fa}\label{figure9}
\end{figure}

\begin{figure}[h]
\centering
\includegraphics[scale=0.9]{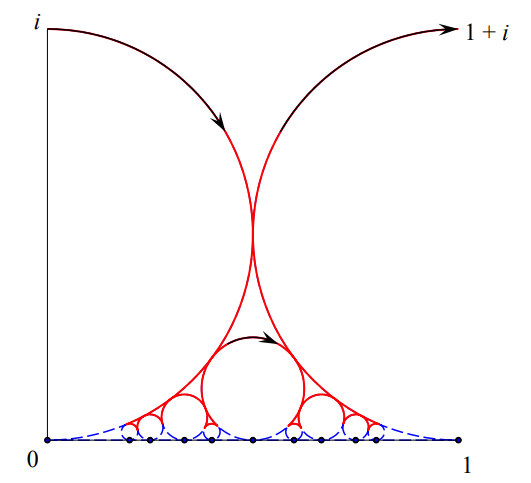}

\caption{The path $P(5)$ for the Farey series $\di F_5:  \frac{0}{1},\frac{1}{5}, \frac{1}{4}, \frac{1}{3}, \frac{2}{5}, \frac{1}{2}, \frac{3}{5},\frac{2}{3},\frac{3}{4},\frac{4}{5}, \frac{1}{1}$.\fa}\label{figure10}
\end{figure}

Let us study the image of $P(N)$ under the mapping $z=e^{2\pi i\tau}$.

\begin{lemma}\label{lemma2}Let $$\mathbb{D}=\left\{z\in\mathbb{C}\,|\,|z|<1\right\}$$ be the unit disc on the complex plane and let $$\mathcal{S}=\left\{\tau\in\mathbb{C}\,|\, 0\leq \text{Re}\,\tau\leq 1, \text{Im}\,\tau>0\right\}$$ be an infinite vertical strip on the upper half plane. 
The transformation 
$$z=e^{2\pi i \tau}$$ is a conformal mapping that maps the vertical strip $\mathcal{S}$ in the $\tau$-plane onto the punctured unit disc $\mathbb{D}\setminus\{0\}$ in the $z$-plane. The mapping is one-to-one on the interior of $\mathcal{S}$. The image of $P(N)$ under this mapping is a simple closed contour $\mathcal{C}$ that enclosed  the origin. 
\end{lemma}
\begin{proof}
\begin{figure}[h]
\centering
\includegraphics[scale=0.9]{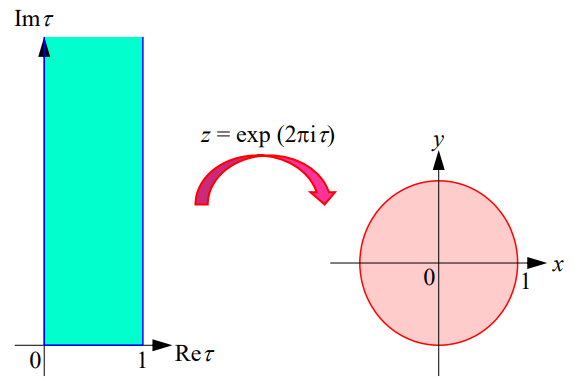}

\caption{The map $z=e^{2\pi i \tau}$.\fa}\label{figure7}
\end{figure}

The mapping 
$$z=e^{2\pi i \tau}$$ is clearly analytic, and so it is a conformal mapping. Let  $z=x+iy=re^{i\theta}$ and $\tau=\sigma+i\lambda$. Then we find that
$$re^{i\theta}=e^{2\pi i \sigma}e^{-2\pi \lambda}.$$ As $\lambda>0$, we find that
$$0<e^{-2\pi\lambda}<1.$$
On the other hand, as $\sigma$ varies from 0 to 1, $\theta=2\pi \sigma$ varies from 0 to $2\pi$. Hence, we see that the mapping transforms $\mathcal{S}$ onto $\mathbb{D}\setminus\{0\}$, and it is one-to-one except on the two vertical lines $\sigma=0$ and $\sigma=1$. 

Let $\mathcal{C}$ be the image of $P(N)$ under the map $z=e^{2\pi i\tau}$.
 Since the initial and end points of $P(N)$ are $i$ and $1+i$, and they are mapped to the same point in the unit disc $\mathbb{D}$, and so $\mathcal{C}$ is a closed curve. Except for the initial and end points, the rest of $P(N)$ lies in the interior of $\mathcal{S}$. Since $P(N)$ is a simple curve and the mapping $z=e^{2\pi i \tau}$ is one-to-one in the interior of $\mathcal{S}$, $\mathcal{C}$ is a simple closed curve.
\end{proof}

To derive the Rademacher's formula, we are extracting the coefficient $p(n)$ from the generating function $F(z)=F(e^{2\pi i \tau})$. As a function of $\tau$, $F(e^{2\pi i\tau})$ is periodic modulo 1. Hence, we can   extend the path $P(N)$ periodically modulo 1.

The Farey fractions in $F_N$ are given by
\[\frac{h_0}{k_0}, \frac{h_1}{k_1}, \ldots, \frac{h_{m_{N}-1}}{k_{m_N}-1}, \frac{h_{m_N}}{k_{m_N}},\]
where
\[\frac{h_0}{k_0}=\frac{0}{1},\quad\frac{h_1}{k_1}=\frac{1}{N},\quad \frac{h_{m_{N}-1}}{k_{m_N}-1}=\frac{N-1}{N}, \quad\frac{h_{m_N}}{k_{m_N}}=\frac{1}{1}.\]
Define
\[  \frac{h_{m_{N}+1}}{k_{m_N}+1}=\frac{h_1}{k_1}+1=\frac{N+1}{N}.\]

 By \eqref{eq3_16_5}, $U_0$ is an arc  joining the point $i$ to 
\[u_1=\frac{N+i}{N^2+1};\]and
$U_{m_N}$ is an arc joining the point 
\[u_2=\frac{N^2-N+1+i}{N^2+1}\] to $1+i$.  Translating $U_0$ by 1 and combine with $U_{m_N}$, we obtain a single arc of the circle $|z-1-i|=1/2$ joining $u_2$ to $u_1+1$.

Henceforth, we regard the path $P(N)$ as a path consists of the union of $m_N$ arcs $\widetilde{U}_1$, $\ldots$, $\widetilde{U}_{m_N}$, 
  corresponding to the fractions $h_j/k_j$ with $1\leq j\leq m_N$.
  \begin{align}\label{eq3_31_10_2}
P(N)=\bigcup_{j=1}^{m_N}\widetilde{U}_j,
\end{align} For $1\leq j\leq m_N-1$, $\widetilde{U}_j=U_j$; whereas $\widetilde{U}_{m_N}$ is the union of $U_{m_N}$ and $U_0+1$ that we have just described. If we take $h_{m_N-1}/k_{m_N-1}$, $h_{m_N}/k_{m_N}$ and $h_{m_N+1}/k_{m_N+1}$ as consecutive Farey fractions, $\widetilde{U}_{m_N}$ is the upper arc of the circle $C(h_{m_N}, k_{m_N})$ that joins the point $\alpha_1(h_{m_N}, k_{m_N})$ to $\alpha_2(h_{m_N}, k_{m_N})$ given by \eqref{eq3_16_5}.

Before ending this section, we study further properties of the arc $\widetilde{U}_j$ in the Ford circle $C(h_j,k_j)$ under the fractional linear transformation 
\[w=-ik^2\left(\tau-\frac{h}{k}\right).\]

\begin{theorem}\label{thm3_23_1}
The mapping
\begin{align}\label{eq3_23_1}
w=-ik^2\left(\tau-\frac{h}{k}\right)
\end{align}transforms the Ford circle $C(h,k)$ in the $\tau$-plane onto a circle $\tilde{C}$ in the $w$-plane. The center of $\tilde{C}$ is the point $w=1/2$ and the radius is $1/2$. If $h_1/k_1, h/k, h_2/k_2$ are three consecutive Farey fractions in  $F_N$,   $\alpha_1(h, k)$ is the point of tangency between $C(h_1, k_1)$ and $C(h,k)$, $\alpha_2(h, k)$ is the point of tangency between $C(h,k)$ and $C(h_2, k_2)$, the images of $\alpha_1(h,k)$ and $\alpha_2(h,k)$ under the mapping \eqref{eq3_23_1} are $w_1(h,k)$ and $w_2(h,k)$ given by
 \begin{equation}\label{eq3_23_2}\begin{split}
w_1(h,k)&=\frac{k^2}{k^2+k_1^2}+i\frac{kk_1}{k^2+k_1^2}, \\
w_2(h,k)&=\frac{k^2}{k^2+k_2^2}-i\frac{kk_2}{k^2+k_2^2}. 
\end{split}\end{equation}The upper arc on $C(h,k)$ joining $\alpha_1(h,k)$ to $\alpha_2(h,k)$ is mapped to the arc on $\tilde{C} $ joining $w_1(h,k)$ to $w_2(h,k)$ that does not touch the imaginary axis.

 \begin{figure}[h]
\centering
\includegraphics[scale=0.9]{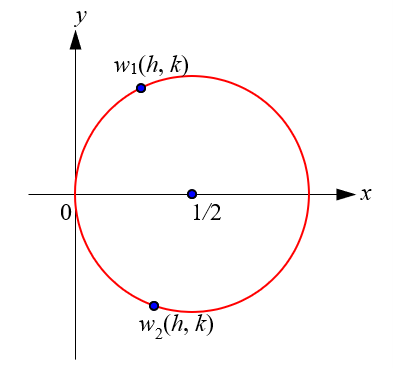}

\caption{The circle $\tilde{C}$ in the $w$-plane, which is the image of $C(h,k)$ in the $\tau$-plane under the transformation $\di w=-ik^2\left(\tau-\frac{h}{k}\right)$.\fa}\label{figure5}
\end{figure}
\end{theorem}
\begin{proof}
The mapping \eqref{eq3_23_1} is a  fractional linear transformation. So it maps a circle to a circle. It is easy to 
verify the assertions in the theorem. In particular, since the mapping maps the real axis to the imaginary axis, the image of the upper arc on $C(h,k)$ joining $\alpha_1(h,k)$ to $\alpha_2(h,k)$, is  the arc on $\tilde{C} $ joining $w_1(h,k)$ to $w_2(h,k)$ that is away from the imaginary axis.
\end{proof}

\begin{proposition}\label{prop4_3_1}Let $\tilde{C}$ be the circle on the complex plane with center at $1/2$ and radius $1/2$. 
For any $z\in \tilde{C}$, $\text{Re}\,z\geq 0$ and $$\text{Re}\,\left(\frac{1}{z}\right)=1.$$
\end{proposition}
\begin{proof}
If $z=x+iy$ is a point on $\tilde{C}$,
$$\left(x-\frac{1}{2}\right)^2+y^2=\frac{1}{4}.$$ This implies that
$$\left|x-\frac{1}{2}\right|\leq\frac{1}{2}$$ and hence $$x\geq 0.$$
On the other hand, $z$ can be written as
$$z=\frac{1}{2}(1+e^{i\theta})$$ for some $\theta\in [0,2\pi]$. It follows that $$\text{Re}\,\left\{\frac{1}{z}\right\}=
\frac{ 2(1+\cos\theta)}{(1+\cos\theta)^2+\sin^2\theta}=1.$$  
\end{proof}
\begin{theorem}\label{thm4_1_4} 
Let $h_1/k_1, h/k, h_2/k_2$ be three consecutive Farey fractions in   $F_N$, and let 
  $w_1(h,k)$ and $w_2(h,k)$ be the points given in \eqref{eq3_23_2} in Theorem \ref{thm3_23_1}. The moduli of these points are given by
\begin{align}\label{eq3_23_3}
|w_1(h,k)|=\frac{k}{\sqrt{k^2+k_1^2}},\hspace{1cm}|w_2(h,k)|=\frac{k}{\sqrt{k^2+k_2^2}}.
\end{align}For any point $w$ on the chord joining $w_1(h,k)$ and $w_2(h,k)$,  
$$|w|\leq\frac{\sqrt{2}k}{N+1}.$$ Hence, the length the chord is not larger than
$2\sqrt{2}k/(N+1)$. On the other hand, the real part of $w$ is bounded below by $k^2/(2N^2)$, and hence,
\begin{align*}
\text{Re}\,\left\{\frac{1}{w}\right\}>  \frac{1}{4}.
\end{align*}
\end{theorem}
\begin{proof}
Eq. \eqref{eq3_23_3} is easily verified. 
To prove the second assertion, notice that if $k+k_1\leq N$, then 
$(h+h_1)/(k+k_1)$ is a Farey fraction in $F_n$ and it lies between $h_1/k_1$ and $h/k$. This contradicts to $h_1/k_1$ and $h/k$ are consecutive in $F_N$. Therefore, we must have $k+k_1\geq N+1$. Similarly, $k+k_2\geq N+1$. Hence, 
\begin{align*}
 k^2+k_1^2 \geq \frac{ (k+k_1)^2}{2}\geq\frac{(N+1)^2}{2}.
\end{align*}Similarly,
$$k^2+k_2^2\geq \frac{(N+1)^2}{2}.$$

If $w$ is on the chord joining $w_1(h,k)$ and $w_2(h,k)$, then there is a real number $t$ in the interval $[0,1]$ such that
\begin{align*}
w=tw_1(h,k)+(1-t)w_2(h,k).
\end{align*}It follows that
\begin{align*}
|w| 
&\leq  t|w_1(h,k)|+(1-t)|w_2(h,k)|\\
&=t\frac{  k}{\sqrt{k^2+k_1^2}}+(1-t)\frac{k}{\sqrt{k^2+k_2^2}}\\
&\leq  \frac{\sqrt{2}tk}{N+1}+\frac{\sqrt{2}(1-t)k}{N+1}\\
&=\frac{\sqrt{2}k}{N+1}.
\end{align*} By triangle inequality, the length of the chord is not larger than $|w_1|+|w_2|$, and hence, it is not larger than $2\sqrt{2}k/(N+1)$. 

Now, for the assertion about the real part, 
we use the fact that $k, k_1, k_2$ are all not larger than $N$. Then
\begin{align*}
\text{Re}\,w&=t\text{Re}\,w_1(h,k)+(1-t)\text{Re}\,w_2(h,k)\\
&=\frac{tk^2}{k^2+k_1^2}+\frac{(1-t)k^2}{k^2+k_2^2}\\
&\geq t\frac{k^2}{2N^2}+(1-t)\frac{k^2}{2N^2}\\
&=  \frac{k^2}{2N^2}.
\end{align*}
It follows that 
\begin{align*}
\text{Re}\,\left\{\frac{1}{w}\right\}=\frac{\text{Re}\,w}{|w|^2}> \frac{k^2}{2N^2}\times \frac{N^2}{2k^2}=\frac{1}{4}.
\end{align*}

\end{proof}
Finally, we have the following elementary proposition.
\begin{proposition}\label{prop4_4_1}
\label{cor4_3_1} Let $C$ be the circle with center at $1/2$ and radius $1/2$. Given a point $w$ on $C$, let  $\mathcal{W}$ be the minor arc joining the point 0 to the point $w$. Then $\ell(\mathcal{W})$, the arc-length of $\mathcal{W}$, is bounded above by
$ \pi|w|/2.$ For any $z$ on $\mathcal{W}$, $|z|\leq |w|$.
\end{proposition}
\begin{proof}
\begin{figure}[h]
\centering
\includegraphics[scale=0.85]{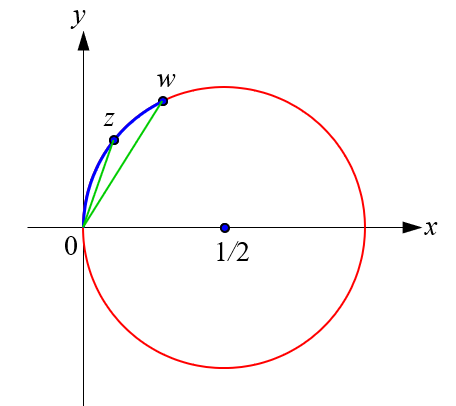}

\caption{The circle $C$.\fa}\label{figure6}
\end{figure}
Any point $z$ on $C$ can be written as
\begin{align}
\label{eq4_4_1} z=\frac{1}{2}\left(1+e^{i\theta}\right)\end{align} for some $\theta\in [0, 2\pi]$. Let 
$$ w=\frac{1}{2}\left(1+e^{i\theta_w}\right).$$ By symmetry of $C$, we can assume that $\theta_w \in [0, \pi)$. 
Then $$|w|=\frac{1}{2}\sqrt{(1+\cos\theta_w)^2+\sin^2\theta_w}=\frac{1}{2}\sqrt{2+2\cos\theta_w}=\cos\frac{\theta_w}{2}.$$
The arc-length of $\mathcal{W}$ is 
$$\ell(\mathcal{W})=\frac{1}{2}\alpha,$$ where  $\alpha=\pi-\theta_w$ is the central angle of  $\mathcal{W}$.
 It is  elementary  to show that if $0< t \leq\pi/2$,
$$\frac{2}{\pi}\leq\frac{\sin t}{t}\leq 1.$$
Therefore,
\begin{align*}
\frac{\ell(\mathcal{W})}{|w|}=\frac{\alpha}{\di 2\sin \frac{\alpha}{2}}\leq\frac{\pi}{2}.
\end{align*}
This proves the first assertion.

Now, if $z$ is a point on $ \mathcal{W}$, then $z$ is given by \eqref{eq4_4_1} for some $\theta\in [\theta_w, \pi]$.
If follows that
$$|z|=\cos\frac{\theta}{2}\leq \cos\frac{\theta_w}{2}\leq |w|.$$

\end{proof}

\bigskip
\section{Rademacher's Convergent Series}

In this section, we will prove the Rademacher's Theorem \ref{thmrademacher}, which we state again here.
\begin{theorem}\label{thmrademacher2}[\textbf{Rademacher's Formula \cite{Rademacher1937}}]~

If $n\geq 1$, the partition function $p(n)$ is represented by the convergent series
\begin{align}\label{eq4_24_2}p(n)=\sum_{k=1}^{\infty}R_k(n),\end{align} where
\begin{align*}
R_k(n)=\frac{1}{\pi\sqrt{2}}A_k(n)\sqrt{k}\frac{d}{dn}\left(\frac{\di\sinh\left\{\frac{\pi}{k}\sqrt{\frac{2}{3}\left(n-\frac{1}{24}\right)}\right\}}{\di\sqrt{n-\frac{1}{24}}}\right),
\end{align*}
\begin{gather*}
A_k(n)=\sum_{\substack{1\leq h\leq k\\(h,k)=1 }}\exp\left(\pi i s(h,k)-\frac{2\pi i nh}{k}\right),\end{gather*} and $s(h,k)$ is the Dedekind sum defined by
\begin{gather}\label{eq3_31_9}
s(h,k)=\sum_{r=1}^{k-1}\frac{r}{k}\left(\frac{hr}{k}-\left\lfloor\frac{hr}{k}\right\rfloor-\frac{1}{2}\right).
\end{gather}
\end{theorem}
We break the proof into a few lemmas.

\begin{lemma}
\label{lemma1}
Let 
\begin{align}\label{eq3_31_7}F(z)=\prod_{m=1}^{\infty}\frac{1}{1-z^m},\end{align} and let $\mathcal{C}$ be any positively oriented simple closed curve in the unit disc $\mathbb{D}$ surrounding the point $z=0$. Given a positive integer integer $n$, 
\begin{align*}
p(n)=\frac{1}{2\pi i}\oint_{\mathcal{C}}\frac{F(z)}{z^{n+1}}dz.
\end{align*}
\end{lemma}
\begin{proof}
The infinite product in \eqref{eq3_31_7} converges absolutely when $|z|<1$. Hence, $F(z)$ is an analytic function on the unit disc $\mathbb{D}$. Since the Taylor series of $F(z)$ is given by
\begin{align*}
F(z)=\sum_{n=0}^{\infty}p(n)z^n.
\end{align*}By Cauchy integral formula,
\begin{align*}
p(n)=\frac{1}{2\pi i}\oint_{\mathcal{C}}\frac{F(z)}{z^{n+1}}dz.
\end{align*}
\end{proof}

As we have mentioned in the proof of Lemma \ref{lemma1}, the function $F(z)$ is analytic in the unit disc $\mathbb{D}$. By the product expansion, we see that it has singularities at the points $e^{2\pi i h/k}$, where $h$ and $k$ are positive integers. 
In the following theorem, we present a transformation formula for $F(z)$ near a singularity point $e^{2\pi i h/k}$.

\begin{theorem}\label{thm4_1_1}
Let \begin{align}\label{eq3_31_7_2}F(z)=\prod_{m=1}^{\infty}\frac{1}{1-z^m},\end{align} and let $h$ and $k$ be integers with $1\leq h\leq k$ and $(h,k)=1$. Let $H$ be a positive integer not larger than $k$ such that $hH\equiv -1\;\text{mod}\;k$.  Finally, let  $z$ be a complex number with $\text{Re}\,(z)>0$. Define
$$w=\exp\left(\frac{2\pi i h}{k}-\frac{2\pi z}{k^2}\right),\hspace{1cm}w'=\exp\left(\frac{2\pi i H}{k}-\frac{2\pi}{z}\right).$$  Then
\begin{align}
F(w)=e^{\pi i s(h,k)}\left(\frac{z}{k}\right)^{\frac{1}{2}}\exp\left(\frac{\pi}{12 z}-\frac{\pi z}{12k^2}\right)F(w'),
\end{align}where $s(h,k)$ is given by \eqref{eq3_31_9}. 
\end{theorem}
When $|z|$ is small, $w$ is a point that is close to the point $e^{2\pi i h/k}$, whereas $w'$ is a point that is close to the origin. Since $F(0)=1$, this theorem says that when $w$ is close to the singularity point $e^{2\pi i h/k}$,
\begin{align*}
F(w)\sim e^{\pi i s(h,k)}\left(\frac{z}{k}\right)^{\frac{1}{2}}\exp\left(\frac{\pi}{12 z} \right).
\end{align*}
\begin{proof}
Recall that 
$\di 
F(e^{2\pi i \tau})= e^{\frac{\pi i \tau}{12}}\eta(\tau)^{-1}$, where $\eta(\tau)$ is the Dedekind eta function \eqref{eq3_31_8}. It follows that
\begin{align*}
F(w)=\frac{e^{\frac{\pi i \tau}{12}}}{\eta(\tau)},
\hspace{1cm}\text{where}\;\;\tau = \frac{h}{k}+i\frac{z}{k^2}.\end{align*}Since
$$hH=-1\mod k,$$ there exists an integer $K$ such that
$$kK-hH=1.$$This implies that \[ \begin{bmatrix} H & -K\\ k & -h\end{bmatrix}\]is an element of the modular group $\Gamma=\text{SL}\,(2,\mathbb{Z})$. 
Notice that 
\begin{align*} 
\tau'&=\frac{H\tau-K}{k\tau-h}\\&=\frac{\di \frac{Hh-Kk}{k}+i\frac{Hz}{k^2}}{\di \frac{iz}{k}}\\
&=\frac{H}{k}+\frac{ i}{z}.
\end{align*}Hence, $w'=e^{2\pi i\tau'}$. The transformation formula for the Dedekind eta function Theorem \ref{thm4_5_2} says that
\begin{align*}
\eta(\tau)^{-1}=\eta(\tau')^{-1}\bigl\{-i(k\tau-h)\bigr\}^{1/2}\exp\left\{\pi i\left(\frac{H-h}{12k}+s(h,k)\right)\right\}.
\end{align*}It follows that
\begin{align*}
F(w)&=e^{\frac{\pi i\tau}{12}}\bigl\{-i(k\tau-h)\bigr\}^{1/2}\exp\left\{\pi i\left(\frac{H-h}{12k}+s(h,k)\right)\right\}e^{-\frac{\pi i \tau'}{12}}F(w')\\
&=\left(\frac{z}{k}\right)^{\frac{1}{2}}\exp\left\{\pi i\left(\frac{H-h}{12k}+s(h,k)+\frac{h-H}{12k}+\frac{iz}{12k^2}-\frac{i}{12z}\right)\right\}F(w')\\
&=e^{\pi i s(h,k)}\left(\frac{z}{k}\right)^{\frac{1}{2}}\exp\left(\frac{\pi}{12 z}-\frac{\pi z}{12k^2}\right)F(w').
\end{align*}
\end{proof}

\begin{proof}
[Proof of Theorem \ref{thmrademacher2}]
Fixed a positive integer $N$. We let $\mathcal{C}$ be the image of the path $P(N)$ defined by \eqref{eq3_31_10_2} under the map
$z=e^{2\pi i \tau}.$

By Lemma \ref{lemma1},
\begin{align*}
p(n)=\frac{1}{2\pi i}\oint_{\mathcal{C}}\frac{F(z)}{z^{n+1}}dz.
\end{align*}Making the change of variables $z=e^{2\pi i \tau}$, we have
\begin{align*}
p(n)&= \int_{P(N)} F(e^{2\pi i \tau})e^{-2\pi i n\tau}d\tau\\
&=\sum_{j=1}^{m_N}\int_{\widetilde{U}_j} F(e^{2\pi i \tau})e^{-2\pi i n\tau}d\tau.
\end{align*} Each $\widetilde{U}_j$ is the upper arc of the circle $C(h_j, k_j)$ with initial point $\alpha_1(h_j, k_j)$ and end point $\alpha_2(h_j, k_j)$, where
\begin{align*}
\alpha_1(h_j,k_j)&=\frac{h_j}{k_j}-\frac{k_{j-1}}{k_{j}(k_{j}^2+k_{j-1}^2)}+\frac{i}{k^2+k_{j-1}^2}, \\
\alpha_2(h_j,k_j)&=\frac{h_j}{k_j}+\frac{k_{j+1}}{k(k^2+k_{j+1}^2)}+\frac{i}{k^2+k_{j+1}^2}. 
\end{align*}

For the integral over the arc $\widetilde{U}_j$, we make a change of variables
$$z=-ik_j^2\left(\tau-\frac{h_j}{k_j}\right)$$ so that
$$\tau=\frac{h_j}{k_j}+i\frac{z}{k_j^2}.$$
Then
\begin{align}\label{eq4_1_2}
\int_{\widetilde{U}_j} F(e^{2\pi i \tau})e^{-2\pi i n\tau}d\tau&=\frac{i}{k^2}\int_{V_j}F\left(\exp\left(\frac{2\pi i h_j}{k_j}-\frac{2\pi z}{k_j^2}\right)\right)e^{-\frac{2\pi i n h_j}{k_j}}e^{\frac{2\pi n z}{k_j^2}}dz.
\end{align}Here  $V_j$ is the arc of the circle $\tilde{C}$ with center at $1/2$ and radius $1/2$, which do not touch the imaginary axis, and joining the points $w_1(h_j, k_j)$ and $w_2(h_j, k_j)$ given by
 \begin{equation}\label{eq4_1_1}\begin{split}
w_1(h_j,k_j)&=\frac{k_j^2}{k_j^2+k_{j-1}^2}+i\frac{k_jk_{j-1}}{k_j^2+k_{j-1}^2}, \\
w_2(h_j,k_j)&=\frac{k_j^2}{k_j^2+k_{j+1}^2}-i\frac{k_jk_{j+1}}{k_j^2+k_{j+1}^2}. 
\end{split}\end{equation}
Now we use the transformation formula given in Theorem \ref{thm4_1_1}, 
\begin{align}\label{eq4_1_3}
F\left(\exp\left(\frac{2\pi i h}{k}-\frac{2\pi z}{k^2}\right)\right)=e^{\pi i s(h,k)}\left(\frac{z}{k}\right)^{\frac{1}{2}}\exp\left(\frac{\pi}{12 z}-\frac{\pi z}{12k^2}\right)
F\left(\exp\left(\frac{2\pi i H}{k}-\frac{2\pi}{z}\right)\right),
\end{align}where   $H$ is a positive integer not larger than $k$ and such that $hH\equiv -1 $ mod $k$.

Let
\begin{align}\label{eq4_1_8}
G(z)=F(z)-1=\sum_{m=1}^{\infty}p(m)z^m.
\end{align}
Then \eqref{eq4_1_2} and \eqref{eq4_1_3} imply that
\begin{align*}
\int_{\widetilde{U}_j} F(e^{2\pi i \tau})e^{-2\pi i n\tau}d\tau&=\mathcal{A}_{h_j, k_j}+\mathcal{B}_{h_j,k_j},
\end{align*}where
\begin{align*}
\mathcal{A}_{h,k}&=\frac{i}{k^{5/2}}\int_{V}e^{\pi i s(h,k)-\frac{2\pi i n h}{k}} z^{\frac{1}{2}}\exp\left(\frac{\pi}{12 z}+\frac{2\pi n z}{k^2}-\frac{\pi z}{12k^2}\right)
 dz,\\
\mathcal{B}_{h,k}&=\frac{i}{k^{5/2}}\int_{V}e^{\pi i s(h,k)-\frac{2\pi i n h}{k}}z^{\frac{1}{2}}\exp\left(\frac{\pi}{12 z}+\frac{2\pi n z}{k^2}-\frac{\pi z}{12k^2}\right)
G\left(\exp\left(\frac{2\pi i H}{k}-\frac{2\pi}{z}\right)\right) dz.
\end{align*}
As $j$ runs from $1$ to $m_N$, $(h, k)$ runs through all pairs of positive integers with $1\leq h\leq k\leq N$ and $(h, k)=1$. Therefore,
\begin{align*}
p(n)=\sum_{k=1}^N\sum_{\substack{1\leq h\leq k\\(h,k)=1 }}\left(\mathcal{A}_{h,k}+\mathcal{B}_{h,k}\right).
\end{align*}
We estimate $\mathcal{B}_{h,k}$ first. Given $h/k $ in $F_N$,   let $W$ be the chord joining the point $w_1(h, k)$ to the point $w_2(h, k)$. Then $V-W$ is  a closed curve that does not enclose the origin. Since the integrand 
\begin{align}
\label{eq4_1_6}
\mathcal{G}(z,h,k,n)=\frac{i}{k^{5/2}} e^{\pi i s(h,k)-\frac{2\pi i n h}{k}}z^{\frac{1}{2}}\exp\left(\frac{\pi}{12 z}+\frac{2\pi n z}{k^2}-\frac{\pi z}{12k^2}\right)
G\left(\exp\left(\frac{2\pi i H}{k}-\frac{2\pi}{z}\right)\right)
\end{align} is analytic in an open set containing the region enclosed by the closed curve $V_j-W_j$, by residue theorem,
we have
$$\oint_{ V-W} \mathcal{G}(z,h,k,n)dz=0.$$ Hence,
\begin{align*}
\mathcal{B}_{h,k}&= \int_{W} \mathcal{G}(z,h,k,n)dz.
\end{align*}

For any $z\in W$, Theorem \ref{thm4_1_4} says that
\begin{align}
\label{eq4_1_7}
|z|\leq \frac{\sqrt{2}k}{N+1},\end{align}
and the arclength of $W$, $\ell(W)$, is not larger than $2\sqrt{2}k/(N+1)$. If we can find an $M$ such that 
\begin{align}\label{eq4_1_5}
\left| \mathcal{G}(z,h,k,n)\right|\leq M\hspace{1cm}\text{for all }\;z\in W,
\end{align}then
\begin{align*}
\left|\mathcal{B}_{h,k}\right|\leq M\ell(W)\leq \frac{2\sqrt{2}k}{N+1}M.
\end{align*}
Let us now find an $M$ satisfying \eqref{eq4_1_5}.  By definition \eqref{eq4_1_6}, 
\begin{align*}
\left|\mathcal{G}(z,h,k,n)\right|=\frac{1}{k^{5/2}}|z|^{\frac{1}{2}}\left|\exp\left(\frac{\pi}{12 z}+\frac{2\pi n z}{k}-\frac{\pi z}{12k^2}\right)\right|
\left|G\left(\exp\left(\frac{2\pi i H}{k}-\frac{2\pi}{z}\right)\right)\right|.
\end{align*}
The estimate of $|z|$ is given by \eqref{eq4_1_7}. Now,
\begin{align*}
&\left|\exp\left(\frac{\pi}{12 z}+\frac{2\pi n z}{k}-\frac{\pi z}{12k^2}\right)\right|\\&= \exp\left(\text{Re}\,\left\{\frac{\pi}{12 z}+\frac{2\pi n z}{k}-\frac{\pi z}{12k^2}\right\}\right) \\
&=\exp\left(\frac{\pi}{12}\text{Re}\,\left\{\frac{1}{z}\right\}\right)  \exp\left(\text{Re}\,\left\{\frac{2\pi n z}{k}\right\}\right)\exp\left(-\text{Re}\,\left\{\frac{\pi z}{12k^2}\right\}\right).
\end{align*}For $z\in W$, Proposition \ref{prop4_3_1} implies that $\text{Re}\,z>0$. Hence,
$$\exp\left(-\text{Re}\,\left\{\frac{\pi z}{12k^2}\right\}\right)\leq 1.$$On the other hand, 
\begin{align*}
\text{Re}\,\left\{\frac{2\pi n z}{k}\right\}\leq \frac{2\pi n}{k}|z|\leq \frac{2 \pi n}{k}\leq 2\pi n.
\end{align*}Therefore,
$$ \exp\left(\text{Re}\,\left\{\frac{2\pi n z}{k}\right\}\right)\leq  \exp\left( 2\pi n\right).$$Next, we estimate
\begin{align*}
\left|G\left(\exp\left(\frac{2\pi i H}{k}-\frac{2\pi}{z}\right)\right)\right|.
\end{align*}Using definition \eqref{eq4_1_8} and triangle inequality, we have
\begin{align*}
\left|G\left(\exp\left(\frac{2\pi i H}{k}-\frac{2\pi}{z}\right)\right)\right|&\leq \sum_{m=1}^{\infty}p(m)\left|\exp\left(\frac{2\pi i mH}{k}-\frac{2\pi m}{z}\right) \right|\\
&=\sum_{m=1}^{\infty}p(m) \exp\left( - 2\pi m \text{Re}\,\left\{\frac{1}{z}\right\}\right).
\end{align*}Therefore,
\begin{align*}
\left|\mathcal{G}(z,h,k,n)\right|\leq \frac{2^{1/4}}{k^2\sqrt{N+1} }\exp\left(  2 \pi n \right)
\sum_{m=1}^{\infty}p(m) \exp\left( -2\pi  \left[ m-\frac{1}{24}\right] \text{Re}\,\left\{\frac{1}{z}\right\}\right).
\end{align*}For a point $z\in W$, Theorem \ref{thm4_1_4} says that
$$  x=\text{Re}\,\left\{\frac{1}{z}\right\}\geq\frac{1}{4}.$$ 
Using the fact that $0<p(m_1)<p(m_2)$ when $m_1<m_2$, we find that
\begin{align*}
0&\leq  \sum_{m=1}^{\infty}p(m) \exp\left( -2\pi  \left[ m-\frac{1}{24}\right] \text{Re}\,\left\{\frac{1}{z}\right\}\right)\\
&=\sum_{m=1}^{\infty}p(24m-1)\exp\left(-\frac{\pi x}{12}(24m-1)\right)\\
&\leq  \sum_{k=1}^{\infty}p(k)\exp\left(-\frac{\pi kx}{12}\right)\\
&=F\left(\exp\left(-\frac{\pi  x}{12}\right)\right)-1.
\end{align*}
Since all the coefficients of $F(z)$ are positive, we find that $0<F(a_1)<F(a_2)$ if $0<a_1<a_2$. Hence,
$$F\left(\exp\left(-\frac{\pi  x}{12}\right)\right)-1\leq F\left(e^{-\frac{\pi}{48}}\right)-1=\mathscr{C}_0.$$It follows that
\begin{align*}
\left|\mathcal{G}(z,h,k,n)\right|\leq \mathscr{C}_0\frac{2^{1/4}}{k^2\sqrt{N+1} }\exp\left(2\pi n\right).
\end{align*}Hence,
\begin{align*}
\left|\sum_{k=1}^N\sum_{\substack{1\leq h\leq k\\(h,k)=1 }} \mathcal{B}_{h,k}\right|
&\leq  \mathscr{C}_0\sum_{k=1}^N\sum_{\substack{1\leq h\leq k\\(h,k)=1 }}\frac{2\sqrt{2}k}{N+1}\times\frac{2^{1/4}}{k^2\sqrt{N+1} }\exp\left(2\pi n\right)\\
&\leq  \frac{\mathscr{C}_1}{(N+1)^{3/2}}\sum_{k=1}^N\sum_{\substack{1\leq h\leq k\\(h,k)=1 }}\frac{1}{k},
\end{align*}where $$\mathscr{C}_1=2^{\frac{7}{4}}\mathscr{C}_0\exp\left(2\pi n\right)$$is a fixed constant that only depends on $n$. 
It is easy to see that
\begin{align*}
0&\leq \sum_{k=1}^N\sum_{\substack{1\leq h\leq k\\(h,k)=1 }}\frac{1}{k}\leq   \sum_{k=1}^N 1=N,
\end{align*}since the number of positive integer $h$ such that $1\leq h\leq k$ and $(h,k)=1$ is not larger than $k$.

Hence, we have shown that
$$\left|\sum_{k=1}^N\sum_{\substack{1\leq h\leq k\\(h,k)=1 }} \mathcal{B}_{h,k}\right|
\leq  \frac{\mathscr{C}_1}{\sqrt{N+1}},$$which is of order $N^{-1/2}$. It follows that
\begin{align*}
p(n)=\sum_{k=1}^N\sum_{\substack{1\leq h\leq k\\(h,k)=1\ }} \mathcal{A}_{h,k}+O^*(\mathscr{C}_1N^{-1/2}).
\end{align*}Here we write
\[f(N)=O^*(g(N))\] to mean 
\[|f(N)|\leq |g(N)|.\]
To compute $\mathcal{A}_{h,k}$, we notice that when $z$ traverses the circle $\tilde{C}$ from $w_1(h,k)$ to $w_2(h,k)$ along the arc that is away from the imaginary axis, it traverses clockwise. We can write this arc as the whole circle $-\tilde{C} $ (the minus sign is for the clockwise orientation) minus the two arcs $-\mathcal{W}_1$ and $\mathcal{W}_2$, where $\mathcal{W}_1$ is the arc from 0 to $w_1(h,k)$ which is clockwise, and $\mathcal{W}_2$ is the arc   0 to  $w_2(h,k)$ which is anticlockwise.
This gives,
\begin{align*}
\mathcal{A}_{h,k}&=-\frac{i}{k^{5/2}}\int_{\tilde{C} }e^{\pi i s(h,k)-\frac{2\pi i n h}{k}} z^{\frac{1}{2}}\exp\left(\frac{\pi}{12 z}+\frac{2\pi n z}{k^2}-\frac{\pi z}{12k^2}\right)dz\\&-\frac{i}{k^{5/2}}\int_{\mathcal{W}_1}e^{\pi i s(h,k)-\frac{2\pi i n h}{k}} z^{\frac{1}{2}}\exp\left(\frac{\pi}{12 z}+\frac{2\pi n z}{k^2}-\frac{\pi z}{12k^2}\right)dz\\&+\frac{i}{k^{5/2}}\int_{\mathcal{W}_2}e^{\pi i s(h,k)-\frac{2\pi i n h}{k}} z^{\frac{1}{2}}\exp\left(\frac{\pi}{12 z}+\frac{2\pi n z}{k^2}-\frac{\pi z}{12k^2}\right)dz\\
&=\mathcal{A}_0(h,k)+\mathcal{A}_1(h,k)+\mathcal{A}_2(h,k).
\end{align*}For any $z\in \mathcal{W}_1\cup\mathcal{W}_2$, Proposition \ref{prop4_3_1} and Proposition \ref{prop4_4_1} imply that
\begin{align*}
\left|z^{\frac{1}{2}}\exp\left(\frac{\pi}{12 z}+\frac{2\pi n z}{k^2}-\frac{\pi z}{12k^2}\right)\right|&\leq \frac{2^{1/4}k^{1/2}}{\sqrt{N+1}}\exp\left(\frac{\pi}{12 }\text{Re}\,\left\{\frac{1}{z}\right\}+\frac{2\pi n }{k^2}\text{Re}\,z-\frac{\pi }{12k^2}\text{Re}\,z\right)\\
&\leq \frac{2^{1/4}k^{1/2}}{\sqrt{N+1}} \exp\left(\frac{\pi}{12 } + 2\pi n \right).
\end{align*}

Theorem \ref{thm4_1_4} and Proposition \ref{prop4_4_1} say that $\ell(\mathcal{W}_1)$ and $\ell(\mathcal{W}_2)$ are bounded above by 
$$\frac{\pi}{2}\times\frac{\sqrt{2} k}{N+1}.$$It follows that
\begin{align*}
&\left|\sum_{k=1}^N\sum_{\substack{1\leq h\leq k\\(h,k)=1 }} \left(\mathcal{A}_1(h,k)+\mathcal{A}_2(h,k)\right)\right|\\&\leq  
\sum_{k=1}^N\sum_{\substack{1\leq h\leq k\\(h,k)=1 }}2\times  \frac{\pi}{2}\times\frac{\sqrt{2} k}{N+1}\times \frac{1}{k^{5/2}} \times\frac{2^{1/4}k^{1/2}}{\sqrt{N+1}} \exp\left(\frac{\pi}{12 } + 2\pi n \right)\\
&\leq \frac{\mathscr{C}_2}{\sqrt{N+1}},
\end{align*}where
$$\mathscr{C}_2= 2^{3/4}\pi \exp\left(\frac{\pi}{12 } + 2\pi n \right).$$
This shows that
\begin{align}\label{eq4_24_1}
p(n)=\sum_{k=1}^N\sum_{\substack{1\leq h\leq k\\(h,k)=1 }} \mathcal{A}_{0}(h,k)+O^*((\mathscr{C}_1+\mathscr{C}_2)N^{-1/2}).
\end{align}Taking $N\rightarrow\infty$ limit, we find that
\begin{align*}
p(n)&=\sum_{k=1}^{\infty}\sum_{\substack{1\leq h\leq k\\(h,k)=1 }} \mathcal{A}_{0}(h,k)
\\&=-i\sum_{k=1}^{\infty} \frac{1}{k^{5/2}}A_k(n)\int_{\tilde{C} } z^{\frac{1}{2}}\exp\left\{\frac{\pi}{12 z}+\frac{2\pi   z}{k^2}\left(n-\frac{1}{24}\right)\right\}dz
\end{align*}
where
$$A_k(n)=\sum_{\substack{1\leq h\leq k\\(h,k)=1 }}e^{\pi i s(h,k)-\frac{2\pi i n h}{k}}.$$
In particular, $A_1(n)=1$.
Making a change of variables
$$w=\frac{1}{z},$$ the circle $\tilde{C}$ is mapped to the line $\text{Re}\,w=1$, with $\text{Im}\,w$ goes from $\infty$ to $-\infty$.
Hence,
\begin{align*}
p(n)=-i\sum_{k=1}^{\infty} \frac{1}{k^{5/2}}A_k(n)\int_{\text{Re}\,w=1} w^{-\frac{5}{2}}\exp\left\{\frac{\pi w}{12  }+\frac{2\pi    }{k^2}\left(n-\frac{1}{24}\right)\frac{1}{w}\right\}dw.
\end{align*}
Making a change of variables
$$t=\frac{\pi w}{12},$$ we find that
\begin{align*}
p(n)=-\frac{i\pi^{\frac{3}{2}}}{24\sqrt{3}}\sum_{k=1}^{\infty} \frac{1}{k^{5/2}}A_k(n)\int_{\text{Re}\,t=\frac{\pi}{12}} t^{-\frac{5}{2}}\exp\left(t+\frac{x_k^2}{4t}\right)dt,
\end{align*}
where
\begin{align*}
x_k&=\frac{\pi}{k}\sqrt{\frac{2}{3}\left(n-\frac{1}{24}\right)}.
\end{align*}
By \eqref{eq4_6_2}, we find that
\begin{align*}
p(n)=\sum_{k=1}^{\infty}R_k(n),\end{align*} where
\begin{align*}
R_k(n)=\frac{ \pi^{\frac{5}{2}}}{12\sqrt{3}}  \frac{1}{k^{5/2}}A_k(n) \left(\frac{x_k}{2}\right)^{-\frac{3}{2}}I_{\frac{3}{2}}(x_k).
\end{align*}By Proposition \ref{prop4_13_1}, 
\begin{align*}
R_k(n)&=\frac{ \pi^{\frac{5}{2}}}{12\sqrt{3}}  \frac{1}{k^{5/2}}A_k(n) \left(\frac{x_k}{2}\right)^{-\frac{3}{2}}\sqrt{\frac{2x_k}{\pi}} \frac{1}{\di \frac{dx_k}{dn}}
\frac{d}{dn}\left(\frac{\di\sinh\left\{\frac{\pi}{k}\sqrt{\frac{2}{3}\left(n-\frac{1}{24}\right)}\right\}}{\di\frac{\pi}{k}\sqrt{\frac{2}{3}\left(n-\frac{1}{24}\right)}}\right)\\
&=\frac{1}{\pi\sqrt{2}}A_k(n)\sqrt{k}\frac{d}{dn}\left(\frac{\di\sinh\left\{\frac{\pi}{k}\sqrt{\frac{2}{3}\left(n-\frac{1}{24}\right)}\right\}}{\di\sqrt{n-\frac{1}{24}}}\right).
\end{align*}Notice that \eqref{eq4_24_1} shows that the series
$\di\sum_{k=1}^{\infty}R_k(n)$ indeed converges to  $p(n)$. This completes the proof of the theorem.

\end{proof}
\begin{remark}
Fixed a positive integer $n$. Given a positive integer $N$,  define
$$Q_N(n)=\sum_{k=N+1}^{\infty} R_k(n), $$  so that
$$p(n)=\sum_{k=1}^{N}R_k(n)+Q_N(n).$$ Namely, $Q_N(n)$ is the remainder if we approximate $p(n)$ by the sum of the first $N$ terms in the series \eqref{eq4_24_2}.

By \eqref{eq4_24_1}
\[|Q_N(n)|\leq\frac{\mathscr{C}}{\sqrt{N}},\]where
\begin{align*}
\mathscr{C}=2^{7/4}\left(F(e^{-\frac{\pi}{48}})-1\right)e^{2\pi n}+2^{3/4}\pi\exp\left(\frac{\pi}{12}+2\pi n\right).
\end{align*}Since $Q_N(n)\rightarrow 0$ as $N\rightarrow\infty$, the series \eqref{eq4_24_2} converges absolutely.

\end{remark}

\section{Leading Asymptotic of the Partition Function}
In the following, we   use the exact formula for $p(n)$ we obtained in the previous section to find the leading asymptotic of $p(n)$ as a function of $n$ when $n\rightarrow\infty$.

\begin{theorem}\label{thm8_29_3}The sequence $\{p(n)\}$ satisfies the following asymptotic formula when  $n\rightarrow \infty$. 
\begin{equation}\label{mainterm}\lim_{n\rightarrow\infty}\frac{p(n)}{\di \frac{1}{4n\sqrt{3}}\exp\left(\pi \sqrt{\frac{2n}{3}}\right)}=1.\end{equation}
\end{theorem}

Given a positive integer $n$, let
\[L(n)=\frac{1}{4n\sqrt{3}}\exp\left(\pi \sqrt{\frac{2n}{3}}\right).\]
Eq. \eqref{mainterm} asserts that
\[\lim_{n\rightarrow \infty} \frac{p(n)}{L(n)}=1.\]To prove this, let
\[S(n)=\sum_{k=2}^{\infty}R_k(n).\]Then by  \eqref{eq4_24_2},
\[p(n)=R_1(n)+S(n).\]
To prove \eqref{mainterm}, it is sufficient to show that
\begin{equation}\label{eq8_29_8}\lim_{n\rightarrow \infty}\frac{R_1(n)}{L(n)}=1\end{equation} and
\begin{equation}\label{eq8_29_7}\lim_{n\rightarrow\infty}\frac{S(n)}{L(n)}=0.\end{equation}

By a straightforward computation, we find that
\begin{equation}\label{eq8_29_5}
R_k(n) =\frac{\pi \sqrt{k}}{3\sqrt{2}\di\sqrt{n-\frac{1}{24}}}A_k(n)\frac{\di\frac{\alpha(n)}{k}\cosh \frac{\alpha(n)}{k}-\sinh \frac{\alpha(n)}{k}}{\alpha(n)^2},
\end{equation}where
\[\alpha(n)=\pi\sqrt{\frac{2}{3}\left(n-\frac{1}{24}\right)}.\]

We first prove \eqref{eq8_29_8}.

\begin{lemma}
\label{lemma8_29_1}
Let $n$ be a positive integer. Then
\[\lim_{n\rightarrow \infty}\frac{R_1(n)}{L(n)}=1\] 
\end{lemma}
\begin{proof}
By \eqref{eq8_29_5},
\[R_1(n)=\frac{1}{4\sqrt{3}\di\left(n-\frac{1}{24}\right)} \left\{\left(1-\frac{1}{\alpha(n)} \right)e^{\alpha(n)}+\left(1+\frac{1}{\alpha(n)} \right)e^{-\alpha(n)}\right\}.\]
Since 
\[\lim_{n\rightarrow\infty}\alpha(n)=\infty,\]
we find that
\[\lim_{n\rightarrow\infty}\frac{R_1(n)}{\di \frac{1}{4\sqrt{3}\di\left(n-\frac{1}{24}\right)} e^{\alpha(n)}}=1.\]
It follows that
\begin{align*}
\lim_{n\rightarrow \infty}\frac{R_1(n)}{L(n)}&=\lim_{n\rightarrow \infty} 
\frac{\di \frac{1}{4\sqrt{3}\di\left(n-\frac{1}{24}\right)} e^{\alpha(n)}}{\di\frac{1}{4n\sqrt{3}}\exp\left(\pi \sqrt{\frac{2n}{3}}\right)}\\
&=\lim_{n\rightarrow\infty}\exp\left(\pi\sqrt{\frac{2}{3}\left(n-\frac{1}{24}\right)}-\pi \sqrt{\frac{2n}{3}}\right).
\end{align*}
Since
\begin{align*}
&\lim_{n\rightarrow \infty}\left(\pi\sqrt{\frac{2}{3}\left(n-\frac{1}{24}\right)}-\pi \sqrt{\frac{2n}{3}}\right)\\
&=\pi\sqrt{\frac{2}{3}}\lim_{n\rightarrow\infty}\frac{\di n-\frac{1}{24}-n}{\di\sqrt{n-\frac{1}{24}}+\sqrt{n}}\\
&=0,
\end{align*}we conclude that
\[\lim_{n\rightarrow \infty}\frac{R_1(n)}{L(n)}=1.\]
\end{proof}

To prove \eqref{eq8_29_7}, we need the following.
\begin{lemma}
\label{lemma8_29_2}
For $u>0$,
\begin{equation}\label{eq8_29_2}
\left|\frac{u\cosh u-\sinh u}{u^2}\right|\leq  \frac{u\cosh u}{2}.
\end{equation}
\end{lemma}

\begin{proof}
Let \[g(u)=u\cosh u-\sinh u,\hspace{1cm}h(u)=u^2.\]
By Cauchy mean value theorem, there exists $v\in (0, u)$ such that
\[\frac{u\cosh u-\sinh u}{u^2}=\frac{g(u)}{h(u)}=\frac{g'(v)}{h'(v)}.\]
Now,
\begin{align*}
g'(u)=u\sinh u,\quad h'(u)=2u.
\end{align*}
Thus,
\begin{align*}
\frac{u\cosh u-\sinh u}{u^2}&=\frac{\sinh v}{2}
\end{align*}
By mean value theorem, there exists $w\in (0, v)$ such that
\[\frac{\sinh v}{v}=\cosh w.\]
Since $\cosh u$ is a strictly increasing function, $\cosh w\leq \cosh v\leq \cosh u$. Thus,
\begin{align*}
\frac{u\cosh u-\sinh u}{u^2}&=\frac{\sinh v}{2}\\
&\leq \frac{v\cosh w}{2}\\
&\leq\frac{u\cosh u}{2}.
 \end{align*}
The assertion follows.
\end{proof}

 \begin{proof}[Proof of Theorem \ref{thm8_29_3}]
To complete the proof of Theorem \ref{thm8_29_3}, it remains to prove \eqref{eq8_29_7}. 

By \eqref{eq8_29_2} and \eqref{eq8_29_5}, and using the fact that $|A_k(n)|\leq k$, we have
\begin{equation*}
|R_k(n)|\leq \frac{\pi \alpha(n)}{6\sqrt{2}\di\sqrt{n-\frac{1}{24}}}\frac{1}{k^{3/2}}\cosh\frac{\alpha(n)}{k}.
\end{equation*}Using the fact that $\cosh u<e^u$ when $u>0$, we have
\begin{equation*}
|R_k(n)| \leq \frac{\pi^2}{6\sqrt{3} }\frac{1}{k^{3/2}}\exp\left(\frac{\alpha(n)}{k}\right).
\end{equation*}
Hence,
\begin{align*}
|S(n)|&\leq\sum_{k=2}^{\infty}|R_k(n)|\\
&\leq \sum_{k=2}^{\infty}\frac{\pi^2}{6\sqrt{3} }\frac{1}{k^{3/2}}\exp\left(\frac{\alpha(n)}{k}\right)\\
&\leq \frac{\pi^2}{6\sqrt{3} }\exp\left(\frac{\alpha(n)}{2}\right)\sum_{k=2}^{\infty}\frac{1}{k^{3/2}}\\
&=C\frac{\pi^2}{6\sqrt{3} }\exp\left(\frac{\alpha(n)}{2}\right),
\end{align*}where
\[C=\zeta\left(\frac{3}{2}\right)-1.\]
It follows that
\begin{align*}
\left|\frac{S(n)}{L(n)}\right|&\leq \frac{2C\pi^2 n}{3  }\exp\left(\frac{\pi}{2}\sqrt{\frac{2}{3}\left(n-\frac{1}{24}\right)}-\pi\sqrt{\frac{2n}{3}}\right)\\
&\leq \frac{2C\pi^2 n}{3  }\exp\left(\frac{\pi}{2}\sqrt{\frac{2}{3}n}-\pi\sqrt{\frac{2n}{3}}\right)\\
&=\frac{2C\pi^2 n}{3  }\exp\left(-\frac{\pi}{2}\sqrt{\frac{2}{3}n}\right).
\end{align*}
It is easy to verify that
\[\lim_{n\rightarrow\infty}\frac{2C\pi^2 n}{3  }\exp\left(-\frac{\pi}{2}\sqrt{\frac{2}{3}n}\right)=\frac{2C\pi^2  }{3  }\lim_{n\rightarrow\infty}\frac{n}{\di \exp\left(\frac{\pi}{2}\sqrt{\frac{2}{3}n}\right)}=0.\]
Hence, 
\[\lim_{n\rightarrow\infty}\frac{S(n)}{L(n)}=0.\]
This completes the proof of \eqref{mainterm}.

\end{proof}

Theorem \ref{thm8_29_3} states that  the ratio of $p(n)$ to $L(n)$ approaches 1 when $n$ approaches infinity. This does not mean when we use $L(n)$ to approximate $p(n)$, the error in the approximation, defined as $p(n)-L(n)$, is small. In fact, this error also tends to infinity when $n$ approaches infinity. What we assert is that
the relative error
\[e(n)=\frac{p(n)-L(n)}{p(n)}\]  goes to 0. The percentage relative error is  defined as
\[\varepsilon(n)=\frac{p(n)-L(n)}{p(n)}\times 100\%.\]

Table \ref{table1} shows some values of $p(n)$ as well as the  percentage relative errors in the approximation of $p(n)$ by $L(n)$. Figure \ref{figure11} depicts the  percentage relative errors graphically.  

\begin{figure}[h]
\centering
\includegraphics[scale=0.85]{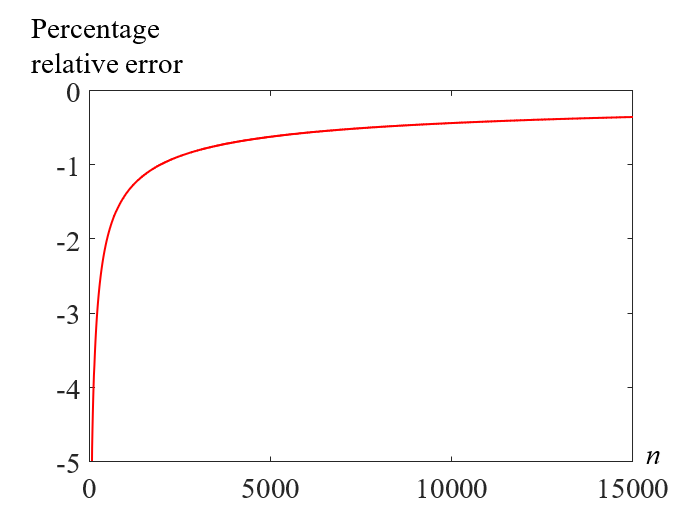}

\caption{The percentage relative error in using $L(n)$ to approximate $p(n)$.\fa}\label{figure11}
\end{figure}
\vfill\pagebreak

\renewcommand{\arraystretch}{1.5}

\begin{center}
\begin{tabularx}
{0.9\textwidth}{
||>{\arraybackslash\hsize=0.6\hsize}c|
>{\arraybackslash\hsize=0.7\hsize}p{20.2cm}||
}
\hline\hline
$n$	&	$p(n)$	 \\
\hline\hline
   10 &                                                                                                                               42\\
\hline
  50&           204226\\
\hline
  100 &        190569292\\\hline
  200 &        3972999029388\\
\hline
  500 &        2300165032574323995027\\
\hline
 1000 &         24061467864032622473692149727991\\
\hline
 2000 &         4720819175619413888601432406799959512200344166\\
\hline
 3000 &         496025142797537184410324879054927095334462742231683423624\\
\hline
 4000 &    1024150064776551375119256307915896842122498030313150910234889093895\\
\hline
 5000 &                           1698201688254421218519751016893064313617576830498292333222038246523\newline 29144349\\
\hline
 6000 & 4671727531970209092971024643973690643364629153270037033856605528925\newline 072405349246129\\
\hline
 7000 & 3285693080344061578628092563592416686195015157453224065969903215743\newline 2236394 
374450791229199\\
\hline
 8000 & 7836026435156834949059314501336459971901076935298586433111860020941\newline 7827764524450990388402844164\\
\hline
 9000& 7713363811780888490732079142740313496163979832207203426264771369460\newline 5367979684296948790335590435626459\\
\hline
10000& 3616725132563629398882047189095369549501603033931565042208186860588\newline
7952568754066420592310556052906916435144\\
\hline
12000&              1294107667757322067493842620367467386268131006205640080126511905905\newline 0170600581269291250270699
01623662251809128853180610\\
\hline
 15000  &2626337936403790841371023191659066988029320559654372494065885879713\newline 75120081791056718639088570913175942816125969709246029351672130266\\
 
\hline\hline
\caption{Selected values of $p(n)$.\fa}
\label{table2}
\end{tabularx}

\end{center}

\vfill\pagebreak
\begin{table}[h]

\renewcommand{\arraystretch}{1.5}

\begin{tabular}{||c|c|c||}
\hline
\hline
$n$	&	$p(n)$	&  percentage relative error $\varepsilon(n)$\\
\hline\hline
10	&$	42	$&	$-14.53$	\\ \hline
50	&$	204226	$&	$-6.54$	\\ \hline
100	&$	190569292	$&	$-4.57$	\\ \hline
200	&$	3.97\times 10^{12}	$&	$-3.2$	\\ \hline
500	&$	2.30 \times 10^{21}	$&	$-2.01$	\\ \hline
1000	&$	2.41\times 10^{31}	$&	$-1.42$	\\ \hline
2000	&$	4.72\times 10^{45}	$&	$-1$	\\ \hline
3000	&$	4.96\times 10^{56}	$&	$-0.81$	\\ \hline
4000	&$	1.02\times 10^{66}	$&	$-0.7$	\\ \hline
5000	&$	1.70\times 10^{74}	$&	$-0.63$	\\ \hline
6000	&$	4.67\times 10^{81}	$&	$-0.57$	\\ \hline
7000	&$	3.29\times 10^{88}	$&	$-0.53$	\\ \hline
8000	&$	7.84\times 10^{94	}$&	$-0.5$	\\ \hline
9000	&$	7.71\times 10^{100}	$&	$-0.47$	\\ \hline
10000	&$	3.62\times 10^{106}	$&	$-0.44$	\\ \hline
12000	&$	1.29\times 10^{117}	$&	$-0.41$	\\ \hline
$\quad 15000\quad$	&$\quad	2.63\times 10^{131}	\quad$&	$-0.36$	\\ \hline
\hline
\end{tabular}

\vspace{0.5cm}

\caption{The values of $p(n)$ and the percentage relative errors in using $L(n)$ to approximate $p(n)$.\fa}
 \label{table1}
\end{table}
 
\vfill\pagebreak

\bibliographystyle{amsalpha}
\bibliography{ref}
\end{document}